\documentclass[11pt,twoside]{amsart}
\usepackage{amsmath, amsthm, amscd, amsfonts, amssymb, graphicx, color}
\usepackage[bookmarksnumbered, colorlinks, plainpages]{hyperref}
\usepackage{pgf,tikz}
\usetikzlibrary{arrows}
\usepackage[utf8]{inputenc}
\usepackage{pstricks-add}
\input{mathrsfs.sty}
\newtheorem{thm}{Theorem}[section]

\newtheorem{prop}[thm]{Proposition}
\newtheorem{defn}[thm]{Definition}
\newtheorem{rem}[thm]{\bf{Remark}}
\newtheorem{example}[thm]{\bf{Example}}

\numberwithin{equation}{section}

\begin{document}


\title{DEGREE POLYNOMIAL OF VERTICES IN A GRAPH AND ITS BEHAVIOR UNDER GRAPH OPERATIONS}
\author{Reza jafarpour-Golzari}
\address{Department of Mathematics, Payame Noor University
,P.O.BOX 19395-3697 Tehran, Iran; Department of Mathematics, Institute for Advanced Studies
in Basic Science (IASBS), P.O.Box 45195-1159, Zanjan, Iran}

\email{r.golzary@iasbs.ac.ir}

\thanks{{\scriptsize
\hskip -0.4 true cm MSC(2010): Primary: 05C07; Secondary: 05C31, 05C76.
\\
\newline\indent{\scriptsize}}}

\maketitle


\begin{abstract}
In this paper, we introduce a new concept namely degree polynomial for vertices of a simple graph.
This notion leads to a concept namely degree polynomial sequence which is stronger than the concept of degree sequence. After obtaining the degree polynomial sequence for some well-known graphs, we prove a theorem which gives a necessary condition for realizability of a sequence of polynomials with coefficients in positive integers. Also we calculate the degree polynomial for vertises of join, Cartesian product, tensor product, and lexicographic product of two simple graphs and for vertices of the complement of a simple graph. Some examples, counterexamples, and open problems concerning to this subjects, is given as well.
\end{abstract}

\vskip 0.8 true cm


\section{\bf Introduction}
\vskip 0.4 true cm

Degree of vertices is one of the essential concepts in theory of graphs with various applications in branches of mathematics and some other fields as network, coding theory and biology (see for instance, \cite{Ber}, \cite{K}, \cite{Lub}, \cite{Mas}). The realizability of a non-increasing sequence of nonnegative integers, means the existence of a simple graph whose degree sequence be the same sequence, is an interesting subject concerning to concept of degree of vertices, with many wide applications (\cite{Aig}, \cite{Ber}, \cite{DuB}, \cite{Er}, \cite{Hak}, \cite{Hav}, \cite{Rit}, \cite{Tri}, \cite{Zre}). P. Er\H{o}ds and T. Gallai in \cite{Er} have presented a theorem which gives a criterion for realizability of a non-increasing sequence of nonnegative integers (Theorem 2.1). Also V. Havel in \cite{Hav} and S. L. Hakimi in \cite{Hak} have presented an algorithm which determines whether such a sequence is realizable or not. A simple graph whose degree sequence be a determined realizable sequence of integers, is not unique under isomorphism, necessarily. But all such graphs are common in some properties.

In this paper we newly introduce a concept, called "degree polynomial", for vertices of a simple graph. This notion leads to concept of degree polynomial sequence. The recent concept is deeper and stronger than the concept of degree sequence in graph theory. Since many topics and various applications arise from the concept of degree sequence, it seems that the new concept, degree polynomial sequence, can cause a wide outlook for future researches.
After obtaining the degree polynomial sequence for some well-known graphs as cycles, complete graphs, complete bipartite graphs, etc, we prove a theorem which gives a necessary condition for realizability of a sequence of polynomials with coefficients in positive integers. Also we study the behavior of degree polynomial, under several graph operations. More precisely, we calculate the degree polynomial for vertices of join, Cartesian product, tensor product, and lexicographic product of two simple graphs and also for vertices of the complement of a simple graph. Some important examples, counterexamples, and open problems are presented, as well.

\vskip 0.8 true cm

\section{\bf Preliminaries}
\vskip 0.4 true cm

In the sequel, we use \cite{Bon} for terminologies and notations on graphs.

Let $G$ be a simple graph. For vertices $u,v\in V(G)$, if $u$ is adjacent with $v$, we write $u\sim v$.

Let $G$ be a simple graph of order $n$. A non-increasing sequence of nonnegative integers $q=(d_{1},\ldots , d_{n})$ is said to be degree sequence of $G$, whenever every degrees of vertices of $G$ be a term of sequence. A sequence $q=(d_{1},\ldots , d_{n})$ of integers is realizable, if there exists a simple graph $G$, such that $q$ be the degree sequence of $G$. Since adding a finite number of isolated vertices to a graph, and deleting a finite number of such vertices from a nonempty graph makes no change in the degree of other vertices, we can consider only the case in which each $d_{i}, 1\leq i \leq n$ is positive.

Let $G$ and $H$ be simple graphs with disjoint vertex sets. The join of $G$ and $H$, denoted by $G\vee H$, is a simple graph with vertex set $V(G)\cup V(H)$, in which for two vertices $u$ and $v$, $u\sim v$ if and only if\\
(1) $u,v\in V(G)$ and $u\sim V$ (in $G$), or\\
(2) $u,v\in V(H)$ and $u\sim V$ (in $H$), or\\
(3) $u\in V(G), \ v\in V(H)$, or\\
(4) $u\in V(H), \ v\in V(G)$.

Let $G$ and $H$ be simple graphs. The Cartesian product of $G$ and $H$, denoted by $G\times H$, is a simple graph with vertex set $V(G)\times V(H)$, in which for two vertices $(u_{1}, v_{1})$ and $(u_{2}, v_{2}), (u_{1}, v_{1})\sim (u_{2}, v_{2})$, if and only if\\
(1) $u_{1}=u_{2}$ and $v_{1}\sim v_{2}$ (in $H$), or\\
(2) $v_{1}=v_{2}$ and $u_{1}\sim u_{2}$ (in $G$).

Also, the tensor product of $G$ and $H$, denoted by $G\otimes H$, is a simple graph with vertex set $V(G)\times V(H)$, in which for two vertices $(u_{1}, v_{1})$ and $(u_{2}, v_{2}), (u_{1}, v_{1})\sim (u_{2}, v_{2})$, if and only if $u_{1}\sim u_{2}$ (in $G$) and $v_{1}\sim v_{2}$ (in $H$). Finally, the lexicographic product of $G$ and $H$, denoted by $G[H]$, is a simple graph with vertex set $V(G)\times V(H)$, in which for two vertices $(u_{1}, v_{1})$ and $(u_{2}, v_{2}), (u_{1}, v_{1})\sim (u_{2}, v_{2})$, if and only if\\
(1) $u_{1}\sim u_{2}$ (in $G$), or\\
(2) $u_{1}=u_{2}$ and $v_{1}\sim v_{2}$ (in $H$).

For a simple graph $G$, the complement of $G$, denoted by $G^{c}$, is a simple graph with vertex set $V(G)$, in which for two vertices $u$ and $v, \ u\sim v$, if and only if $u$ is not adjacent with $v$ in $G$.

Let $K$ be a field and $S=K[x_{1}, \ldots, x_{n}]$ be the polynomial ring in $n$ variables with coefficients in $K$ and let $f\in S$ be any nonzero polynomial. The polynomial $f$ is represented uniquely in the form:
\[f=\sum_{a_{u}(\neq 0)\in K}a_{u}u\]
in terms of distinct monomials $u$ in $S$ where the set of $u$'s is the support of $f$ \cite{Her}. We will call $a_{u}u$'s, the essential terms of $f$.

\vskip 0.8 true cm

\section{\bf Degree sequence}
\vskip 0.4 true cm

Two non-isomorphic simple graphs can have the same degree sequence. Consider the graphs below for example.
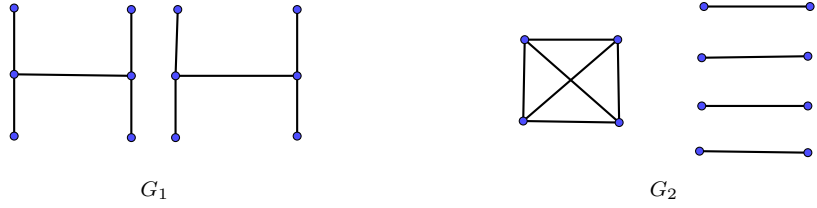
\begin{center}
\definecolor{ududff}{rgb}{0.30196078431372547,0.30196078431372547,1.}
\begin{tikzpicture}[line cap=round,line join=round,>=triangle 45,x=2.0cm,y=1.0cm]
\clip(-3.85,-2.8) rectangle (3.,1.6);
\draw [line width=0.8pt] (-3.4462095730918487,1.1426511627906977)-- (-3.4462095730918487,0.2606046511627907);
\draw [line width=0.8pt] (-3.4462095730918487,0.2606046511627907)-- (-3.4462095730918487,-0.5613023255813954);
\draw [line width=0.8pt] (-2.6673221216041387,1.1226046511627907)-- (-2.6673221216041396,0.24055813953488367);
\draw [line width=0.8pt] (-2.6673221216041396,0.24055813953488367)-- (-2.6673221216041387,-0.5813488372093023);
\draw [line width=0.8pt] (-2.6673221216041396,0.24055813953488367)-- (-3.4462095730918487,0.2606046511627907);
\draw [line width=0.8pt] (-2.358706338939197,1.1226046511627907)-- (-2.373402328589908,0.2405581395348837);
\draw [line width=0.8pt] (-2.373402328589908,0.2405581395348837)-- (-2.373402328589908,-0.5813488372093023);
\draw [line width=0.8pt] (-1.5651228978007752,1.1226046511627907)-- (-1.5651228978007752,0.2405581395348837);
\draw [line width=0.8pt] (-1.5651228978007752,0.2405581395348837)-- (-1.5651228978007752,-0.5613023255813954);
\draw [line width=0.8pt] (-1.5651228978007752,0.2405581395348837)-- (-2.373402328589908,0.2405581395348837);
\draw [line width=0.8pt] (-0.0524838292367391,0.7216744186046511)-- (0.5657956015523941,0.7216744186046511);
\draw [line width=0.8pt] (0.5657956015523941,0.7216744186046511)-- (0.5751875808538171,-0.38088372093023254);
\draw [line width=0.8pt] (0.5751875808538171,-0.38088372093023254)-- (-0.0624838292367391,-0.36083720930232555);
\draw [line width=0.8pt] (-0.0524838292367391,0.7216744186046511)-- (-0.0624838292367391,-0.36083720930232555);
\draw [line width=0.8pt] (-0.0524838292367391,0.7216744186046511)-- (0.5751875808538171,-0.38088372093023254);
\draw [line width=0.8pt] (0.5657956015523941,0.7216744186046511)-- (-0.0624838292367391,-0.36083720930232555);
\draw [line width=0.8pt] (1.8443467011642956,1.1626976744186046)-- (1.138939197930143,1.1626976744186046);
\draw [line width=0.8pt] (1.829650711513584,0.5011627906976744)-- (1.1242432082794316,0.4811162790697674);
\draw [line width=0.8pt] (1.829650711513584,-0.1603720930232558)-- (1.1242432082794316,-0.1603720930232558);
\draw [line width=0.8pt] (1.829650711513584,-0.781813953488372)-- (1.10954721862872,-0.7617674418604651);
\begin{scriptsize}
\draw [fill=ududff] (-3.4462095730918487,1.1426511627906977) circle (1.5pt);
\draw [fill=ududff] (-3.4462095730918487,0.2606046511627907) circle (1.5pt);
\draw [fill=ududff] (-3.4462095730918487,-0.5613023255813954) circle (1.5pt);
\draw [fill=ududff] (-2.6673221216041387,1.1226046511627907) circle (1.5pt);
\draw [fill=ududff] (-2.6673221216041396,0.24055813953488367) circle (1.5pt);
\draw [fill=ududff] (-2.6673221216041387,-0.5813488372093023) circle (1.5pt);
\draw [fill=ududff] (-2.358706338939197,1.1226046511627907) circle (1.5pt);
\draw [fill=ududff] (-2.373402328589908,0.2405581395348837) circle (1.5pt);
\draw [fill=ududff] (-2.373402328589908,-0.5813488372093023) circle (1.5pt);
\draw [fill=ududff] (-1.5651228978007752,1.1226046511627907) circle (1.5pt);
\draw [fill=ududff] (-1.5651228978007752,0.2405581395348837) circle (1.5pt);
\draw [fill=ududff] (-1.5651228978007752,-0.5613023255813954) circle (1.5pt);
\draw [fill=ududff] (-0.0524838292367391,0.7216744186046511) circle (1.5pt);
\draw [fill=ududff] (0.5657956015523941,0.7216744186046511) circle (1.5pt);
\draw [fill=ududff] (0.5751875808538171,-0.38088372093023254) circle (1.5pt);
\draw [fill=ududff] (-0.0624838292367391,-0.36083720930232555) circle (1.5pt);
\draw [fill=ududff] (1.8443467011642956,1.1626976744186046) circle (1.5pt);
\draw [fill=ududff] (1.138939197930143,1.1626976744186046) circle (1.5pt);
\draw [fill=ududff] (1.829650711513584,0.5011627906976744) circle (1.5pt);
\draw [fill=ududff] (1.1242432082794316,0.4811162790697674) circle (1.5pt);
\draw [fill=ududff] (1.829650711513584,-0.1603720930232558) circle (1.5pt);
\draw [fill=ududff] (1.1242432082794316,-0.1603720930232558) circle (1.5pt);
\draw [fill=ududff] (1.829650711513584,-0.781813953488372) circle (1.5pt);
\draw [fill=ududff] (1.10954721862872,-0.7617674418604651) circle (1.5pt);
\draw (-2.65954721862872,-1.0617674418604651) node[anchor=north west] {$G_{1}$};
\draw (0.7254721862872,-1.0617674418604651) node[anchor=north west] {$G_{2}$};
\end{scriptsize}
\end{tikzpicture}
  \end{center}
  \label{graph 1}

A realizable sequence of integers is also named a graphical sequence. If  $q=(d_{1},\ldots , d_{n})$ be a graphical sequence realized by the graph $G$, by elementary properties of simple graphs, we have:

(1) The sum of $d_{i}$'s, $1\leq i\leq n$, is even. Therefore the number of  odd $d_{i}$'s is even and thus the number of even $d_{i}$'s is even if and only if $n$ is even.

(2) For every $1\leq i\leq n$, $d_{i}\leq n-1$.

(3) If $G$ has no isolated vertices, then $2[\frac{n+1}{2}]\leq \sum^{n}_{i=1}d_{i}\leq n(n-1)$.

(4) If no $d_{i}$ be zero, then at lest two of them are equal
(by the pigeonhole principle).

(3) and (4) are implied from (2). 
\begin{thm}
(Er\H{o}ds-Gallai Theorem \cite{Er}) A non-increasing sequence $\phi=(a_{1},\ldots , a_{n})$ of nonnegative integers is graphic if and only if \\
a) $\sum^{n}_{i=1}a_{i}$ be even,\\
b) $\sum^{j}_{i=1}a_{i}- j(j-1)\leq\sum^{n}_{k=j+1}min (j,a_{k}) \ \ \ \  (j=1,\ldots , n-1)$.
\end{thm}

\vskip 0.8 true cm

\section{\bf Degree polynomial}
\vskip 0.4 true cm

Before all, we introduce some notation for convenience. For a polynomial $f(x)=\sum^{n}_{i=1}a_{i}x^{i}\in\mathbb{R}[x]$ with $a_{n}\neq0$, we show the sum of $a_{i}$'s for $1\leq i\leq n$, by $sc(f)$. Also $sec(f)$ and $soc(f)$, use for the sum of $a_{i}$'s for even $i$, and sum of $a_{i}$'s for odd $i$, respectively. We define $sc(0)=0$, as well.

We introduce a total order $<_{pol}$
on the set of all nonzero polynomials with coefficients in nonnegative integers such that $<_{pol}$ compares two distinct polynomials $f= \sum^{n}_{i=0}a_{i}x^{i}$ and $g=\sum^{m}_{i=0}b_{i}x^{i}$ with coefficients in nonnegative integers and with $a_{n},b_{m}\neq0$, as follows:

If $sc(f)\neq sc(g)$, then each one of $f$ and $g$ whose sum of coefficients is greater (as an integer), will be greater;

If $sc(f)=sc(g)$, then supposing that $i_{1}=\max\{i| \ a_{i}b_{i}\neq0\}$, if $a_{i_{1}}\neq b_{i_{1}}$, then each one of $f$ and $g$ which has greater coefficient in $x^{i_{1}}$, will be greater; 

If $sc(f)=sc(g)$ and $a_{i_{1}}=b_{i_{1}}$, then supposing that $i_{2}=\max\{i| \ i<i_{1},\ a_{i}b_{i}\neq0\}$, if $a_{i_{2}}\neq b_{i_{2}}$, then each one of $f$ and $g$ which has greater coefficient in $x^{i_{2}}$, will be greater;

Continue on.

For example,
\[ 2x^{4}+12x^{3}>_{pol} 3x^{5}+x^{2},\]
\[2x^{4}+12x^{2}<_{pol} 2x^{5}+12x^{2}, \ x^{5}+13x^{2},\]
\[2x^{4}+12x^{2}>_{pol} 2x^{4}+11x^{2}+x.\]

\vskip 0.2 true cm

Let $f=\sum_{a_{i}\neq 0}a_{i}x^{i}$ be a nonzero polynomial in $\mathbb{R}[x]$ with coefficients in nonnegative integers where $a_{i}x^{i}$'s are the essential terms of $f$. For $n\in \mathbb{N}$, we denote the polynomial $\sum_{a_{i}\neq 0}a_{i}x^{in}$ by $f^{\curlywedge\times n}$. Also we set $0^{\curlywedge\times n}=0$. If $\deg f\leq n$, we denote the polynomial $\sum_{a_{i}\neq 0} a_{i} x^{n-i}$ by $f^{\curlywedge n-}$. Also we set $0^{\curlywedge n-}=0$.
\begin{defn}
Let $f=\sum_{a_{i}\neq 0}a_{i}x^{i}$, $g=\sum_{b_{j}\neq 0}b_{j}x^{j}$ be two nonzero polynomials in $\mathbb{R}[x]$ with coefficients in nonnegative integers where $a_{i}x^{i}$'s and $b_{j}x^{j}$'s are the essential terms of $f$ and $g$, respectively. The tensor product of $f$ and $g$, denoted by $f\otimes g$, is the polynomial $\sum c_{t}x^{t}$ in which $t$'s are the distinct products of $i$'s and $j$'s, and
\[c_{t}= \sum_{i.j=t}a_{i}b_{j}.\]
Also we set $0\otimes f=0, f\otimes 0=0, 0\otimes f=0$, where $0$ is the zero polynomial.
\end{defn}
Under the conditions of definition 3.1., it is observed simply that, first, $f\otimes g$ can be achieved by tensor-multiplying the essential terms of $f$ by the essential terms of $g$, one by one, and secondly for each $f$ and $g$ with variable $x$ and coefficients in nonnegative integers, $f\otimes g=g\otimes f$.

Now we introduce a new concept, that is the concept of degree polynomial.
\begin{defn}
Let $G$ be a simple graph. For a vertex $v$ of $G$, the degree polynomial of $v$ denoted by $\rm{dp}(v)$, is a polynomial with coefficients in nonnegative integers, in which the coefficient of $x^{i}$ is the number of neighbors of $v$ each of degree $i$; Especially, for an isolated vertex $v$, $\rm{pd}(v)=0$. 
\end{defn}
\begin{example}
Let $G$ be the simple graph with following representation.
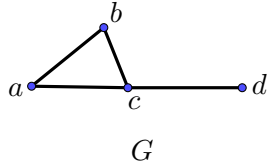
\begin{center}
\definecolor{ududff}{rgb}{0.30196078431372547,0.30196078431372547,1.}
\begin{tikzpicture}[line cap=round,line join=round,>=triangle 45,x=1.0cm,y=1.0cm]
\clip(5.3,1.9999999999998) rectangle (9.56,5.3);
\draw [line width=1.2pt] (7.06,4.18)-- (6.1,3.4);
\draw [line width=1.2pt] (6.1,3.4)-- (7.38,3.38);
\draw [line width=1.2pt] (7.06,4.18)-- (7.38,3.38);
\draw [line width=1.2pt] (7.38,3.38)-- (8.9,3.38);
\draw (5.65,3.58) node[anchor=north west] {$a$};
\draw (7.0,4.66) node[anchor=north west] {$b$};
\draw (7.24,3.4) node[anchor=north west] {$c$};
\draw (8.89,3.7) node[anchor=north west] {$d$};
\draw (7.28,2.82) node[anchor=north west] {$G$};
\begin{scriptsize}
\draw [fill=ududff] (7.06,4.18) circle (1.5pt);
\draw [fill=ududff] (6.1,3.4) circle (1.5pt);
\draw [fill=ududff] (7.38,3.38) circle (1.5pt);
\draw [fill=ududff] (8.9,3.38) circle (1.5pt);
\end{scriptsize}
\end{tikzpicture}
  \end{center}
  \label{graph 2}
\end{example}
We have: 
\[\rm{dp}(a)=x^{2}+x^{3},\]
\[\rm{dp}(b)=x^{2}+x^{3},\]
\[\rm{dp}(c)=2x^{2}+x,\]
\[\rm{dp}(d)=x^{3}.\]

\vskip 0.2 true cm

Since adding a finite number of isolated vertices to a simple graph, and deleting a finite number of such vertices from a nonempty simple graph makes no change in the degree polynomials of other vertices, we will consider only the graphs which has no isolated vertices.
\begin{defn}
For a simple graph $G$, the degree polynomial of $G$, denoted by $\rm{dp}(G)$, is the polynomial $\sum_{i}t_{i}x^{i}$ in $\mathbb{R}[x]$, in which $t_{i}$ is the number of vertices of $G$, each of degree $i$ (specially, $t_{0}$ is the number of isolated vertices of $G$). If $\Delta$ be the maximum degree of $G$, $\rm{dp}(G)$ is of degree $\Delta$.
\end{defn}
It is obvious that if $n$ be the order of $G$, then the sum of all coefficients of $\rm{dp}(G)$ equals $n$. Also, if $m$ be the size of $G$, then the sum of all coefficients of the derivative of $f$ (with respect to $x$) equals $2m$.
\begin{defn}
For a simple graph $G$ of order $n$ without any isolated vertex, a sequence $q=(f_{1},f_{2},\ldots , f_{n})$ of polynomials is said to be the degree polynomial sequence of $G$, if\\
(a) $f_{1}\geq_{pol} \ldots \geq_{pol}f_{n}$,\\
(b) each degree polynomial of vertices of $G$, be a term of sequence.
\begin{example}
For the graph $G$ in Example 3.2, the degree polynomial sequence is:
\[2x^{2}+x, x^{2}+x^{3}, x^{2}+x^{3}, x^{3}.\]
\end{example}
\begin{prop}
Let $G$ be a nonempty simple graph. $G$ is r-regular, if and only if each term of the degree polynomial sequence of $G$ be in the form $rx^{r}$.
\end{prop}
\begin{proof}
Let $G$ be r-regular. For $v\in V(G)$, Since each neighbor of $v$ is of degree $r$, $dp(v)$ is in the form $kx^{r}$. Since $v$ itself has exactly $r$ neighbor, $k=r$.

Conversely, let each term of the degree polynomial sequence of $G$ be $rx^{r}$. Thus each $v \in V(G)$ has exactly $r$ neighbor of degree $r$, and has no another neighbor. Therefore each $v \in V(G)$ is of degree $r$, and $G$ is r-regular.
\end{proof}
If $G$ be a nontrivial complete graph, $K_{n}$, it is obvious that the degree polynomial sequence of $G$ is:
\[(n-1)x^{n-1}, \ldots , (n-1)x^{n-1}\]
where the number of terms is $n$.

If $G$ be a path with $n$ vertices, $P_{n}$, then if $n=2$, the degree polynomial sequence of $G$ is:
\[x, x,\]
if $n=3$, that will be:
\[2x, x^{2}, x^{2},\]
if $n=4$, that will be:
\[x+x^{2}, x+x^{2}, x^{2}, x^{2},\]
and finally, if $n\geq 5$, that will be:
\[2x^{2}, \ldots , 2x^{2}, x+x^{2}, x+x^{2}, x^{2}, x^{2}\]
where the number of terms $2x^{2}$, is $n-4$.

If $G$ be a cycle $C_{n} \ (n\geq 3)$, then, the degree polynomial sequence of $G$ is:
\[2x^{2}, \ldots , 2x^{2},\]
where the number of terms is $n$.

If $G$ be a complete bipartite graph, $K_{r,s}$ where $r\geq s$, then the degree polynomial sequence of $G$ is:
\[rx^{s}, \ldots , rx^{s}, sx^{r},\ldots , sx^{r},\]
where $s$ terms are $rx^{s}$ and $r$ terms are $sx^{r}$.
\end{defn}
\begin{rem}
Supposing $q=(f_{1}, \ldots , f_{n})$ be the degree polynomial sequence of a simple graph $G$, if $f_{i}$ be the degree polynomial of vertex $v_{i}$, since $sc(f_{i})$ be the degree of $v_{i}$, the degree sequence for $G$ is:
\[sc(f_{1}), \ldots , sc(f_{n}).\]
Consequently, if a sequence $q=(f_{1},\ldots , f_{n})$ of nonzero polynomials be realizable, then the sequence
\[sc(f_{1}), sc(f_{2}), \ldots , sc(f_{n})\]
 of integers is realizable. The following example shows that the inverse case is not established. 
\end{rem}
\begin{example}
Consider the sequence:
\[2x, x^{2}, x, x, x\]
of nonzero polynomials with coefficients in nonnegative integers. Although the sequence:
\[sc(2x), sc(x^{2}), sc(x), sc(x), sc(x)\]
that is the sequence:
\[2,1,1,1,1,\]
is realized by simple graph with representation:

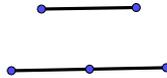
\begin{center}
\definecolor{ududff}{rgb}{0.30196078431372547,0.30196078431372547,1.}
\definecolor{cqcqcq}{rgb}{0.7529411764705882,0.7529411764705882,0.7529411764705882}
\begin{tikzpicture}[line cap=round,line join=round,>=triangle 45,x=1.0cm,y=1.0cm]
\clip(5.8,3.) rectangle (9.,5.);
\draw [line width=1.2pt,] (7.98,4.44)-- (6.72,4.42);
\draw [line width=1.2pt,] (8.38,3.64)-- (7.36,3.62);
\draw [line width=1.2pt,] (7.36,3.62)-- (6.32,3.6);
\begin{scriptsize}
\draw [fill=ududff] (4.16,4.88) circle (2.5pt);
\draw [fill=ududff] (7.98,4.44) circle (1.5pt);
\draw [fill=ududff] (6.72,4.42) circle (1.5pt);
\draw [fill=ududff] (8.38,3.64) circle (1.5pt);
\draw [fill=ududff] (7.36,3.62) circle (1.5pt);
\draw [fill=ududff] (6.32,3.6) circle (1.5pt);
\end{scriptsize}
\end{tikzpicture}
\end{center}
\label{graph 2}
but the sequence
\[2x, x^{2}, x, x, x\]
is not realizable, by Theorem 3.10 (part c).
\end{example}
The following example shows that two simple graphs with the same degree sequence, can have different degree polynomial sequences.
\begin{example}
Consider the graphs $G_{1}$ and $G_{2}$ with following representations.
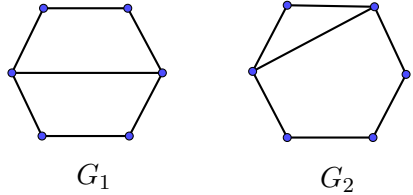
\begin{center}
\definecolor{ududff}{rgb}{0.30196078431372547,0.30196078431372547,1.}
\begin{tikzpicture}[line cap=round,line join=round,>=triangle 45,x=1.0cm,y=1.0cm]
\clip(3.7,2.) rectangle (9.6,6.);
\draw [line width=0.8pt] (8.76,4.84)-- (7.6,4.86);
\draw [line width=0.8pt] (7.6,4.86)-- (7.14,3.98);
\draw [line width=0.8pt] (7.14,3.98)-- (7.6,3.1);
\draw [line width=0.8pt] (7.6,3.1)-- (8.74,3.1);
\draw [line width=0.8pt] (8.74,3.1)-- (9.18,3.94);
\draw [line width=0.8pt] (9.18,3.94)-- (8.76,4.84);
\draw [line width=0.8pt] (5.94,3.96)-- (5.48,4.82);
\draw [line width=0.8pt] (5.48,4.82)-- (4.36,4.82);
\draw [line width=0.8pt] (4.36,4.82)-- (3.94,3.96);
\draw [line width=0.8pt] (3.94,3.96)-- (4.34,3.12);
\draw [line width=0.8pt] (4.34,3.12)-- (5.5,3.12);
\draw [line width=0.8pt] (5.94,3.96)-- (5.5,3.12);
\draw [line width=0.8pt] (8.76,4.84)-- (7.14,3.98);
\draw [line width=0.8pt] (5.94,3.96)-- (3.94,3.96);
\draw (4.66,2.88) node[anchor=north west] {$G_{1}$};
\draw (7.9,2.84) node[anchor=north west] {$G_{2}$};
\begin{scriptsize}
\draw [fill=ududff] (8.76,4.84) circle (1.5pt);
\draw [fill=ududff] (7.6,4.86) circle (1.5pt);
\draw [fill=ududff] (7.14,3.98) circle (1.5pt);
\draw [fill=ududff] (7.6,3.1) circle (1.5pt);
\draw [fill=ududff] (8.74,3.1) circle (1.5pt);
\draw [fill=ududff] (9.18,3.94) circle (1.5pt);
\draw [fill=ududff] (5.94,3.96) circle (1.5pt);
\draw [fill=ududff] (5.48,4.82) circle (1.5pt);
\draw [fill=ududff] (4.36,4.82) circle (1.5pt);
\draw [fill=ududff] (3.94,3.96) circle (1.5pt);
\draw [fill=ududff] (4.34,3.12) circle (1.5pt);
\draw [fill=ududff] (5.5,3.12) circle (1.5pt);
\end{scriptsize}
\end{tikzpicture}
\end{center}
\label{graph 2}
For both $G_{1}$ and $G_{2}$, the degree sequence is $3, 3, 2 , 2, 2, 2.$ But the degree polynomial sequence for $G_{1}$, is:
\[2x^{2}+x^{3}, 2x^{2}+x^{3}, x^{2}+x^{3}, x^{2}+x^{3}, x^{2}+x^{3}, x^{2}+x^{3},\]
and for $G_{2}$, is:
\[2x^{2}+x^{3}, 2x^{2}+x^{3}, 2x^{3}, x^{2}+x^{3}, x^{2}+x^{3}, 2x^{2}.\]
\end{example}
Two non-isomorphic graphs can have the same degree polynomial sequences, as the following example shows.
\begin{example}
Consider two graphs, $G_{1}$ and $G_{2}$, with the following representations:

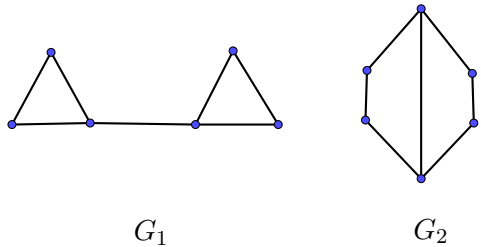
\begin{center}
\definecolor{ududff}{rgb}{0.30196078431372547,0.30196078431372547,1.}
\begin{tikzpicture}[line cap=round,line join=round,>=triangle 45,x=1.0cm,y=1.0cm]
\clip(2.5,2.) rectangle (9.,5.7);
\draw [line width=0.8pt] (8.72,4.56)-- (8.04,5.42);
\draw [line width=0.8pt] (8.04,5.42)-- (7.32,4.6);
\draw [line width=0.8pt] (7.32,4.6)-- (7.3,3.94);
\draw [line width=0.8pt] (7.3,3.94)-- (8.04,3.16);
\draw [line width=0.8pt] (8.04,3.16)-- (8.74,3.9);
\draw [line width=0.8pt] (8.74,3.9)-- (8.72,4.56);
\draw [line width=0.8pt] (6.14,3.88)-- (5.54,4.86);
\draw [line width=0.8pt] (3.12,4.84)-- (2.6,3.88);
\draw [line width=0.8pt] (2.6,3.88)-- (3.64,3.9);
\draw [line width=0.8pt] (3.64,3.9)-- (5.04,3.88);
\draw [line width=0.8pt] (6.14,3.88)-- (5.04,3.88);
\draw (4.08,2.74) node[anchor=north west] {$G_{1}$};
\draw (7.8,2.76) node[anchor=north west] {$G_{2}$};
\draw [line width=0.8pt] (3.64,3.9)-- (3.12,4.84);
\draw [line width=0.8pt] (5.04,3.88)-- (5.54,4.86);
\draw [line width=0.8pt] (8.04,5.42)-- (8.04,3.16);
\begin{scriptsize}
\draw [fill=ududff] (8.72,4.56) circle (1.5pt);
\draw [fill=ududff] (8.04,5.42) circle (1.5pt);
\draw [fill=ududff] (7.32,4.6) circle (1.5pt);
\draw [fill=ududff] (7.3,3.94) circle (1.5pt);
\draw [fill=ududff] (8.04,3.16) circle (1.5pt);
\draw [fill=ududff] (8.74,3.9) circle (1.5pt);
\draw [fill=ududff] (6.14,3.88) circle (1.5pt);
\draw [fill=ududff] (5.54,4.86) circle (1.5pt);
\draw [fill=ududff] (3.12,4.84) circle (1.5pt);
\draw [fill=ududff] (2.6,3.88) circle (1.5pt);
\draw [fill=ududff] (3.64,3.9) circle (1.5pt);
\draw [fill=ududff] (5.04,3.88) circle (1.5pt);
\end{scriptsize}
\end{tikzpicture}
\end{center}
\label{graph 2}
It is obvious that $G_{1}$ and $G_{2}$ are not isomorphic. However, the degree polynomial sequence for both of them is:
\[2x^{2}+x^{3}, 2x^{2}+x^{3}, x^{2}+x^{3}, x^{2}+x^{3}, x^{2}+x^{3}, x^{2}+x^{3},\]
\end{example}
Now we prove a theorem which gives a necessary condition for realizability of a sequence of polynomials with coefficients in nonnegative integers. 
\begin{thm}
If $G$ be a simple graph without any isolated vertices, and $q=(f_{1}, \ldots , f_{n})$ where $f_{1}\geq_{pol} \ldots \geq_{pol} f_{n}$ be the degree polynomial sequence of $G$, then,
\\
(a) $\sum^{n}_{i=1}sc(f_{i})$ is even,
\\
(b) for each nonzero coefficient $k$ of a term $kx^{i}$ in the degree polynomial of a vertex $v$, there are at least $k$ distinct vertices $v_{1}, \ldots , v_{k}$, all distinct from $v$, such that:
\[sc(\rm{dp}(v_{1}))= \ldots = sc(\rm{dp}(v_{k}))= i,\]
\\
(c) $\sum_{sc(f_{j}) \ is \ odd} sec(f_{j})$ and $\sum_{sc(f_{j}) \ is \ even} sec(f_{j})$ are even. 
\end{thm}
\begin{proof}
(a) Let $f_{i}=\rm{dp}(v_{i})$, for $1\leq i\leq n$. We have $\sum^{n}_{i=1}sc(f_{i})=\sum^{n}_{i=1}\deg(v_{i})$. Thus $\sum^{n}_{i=1}sc(f_{i})$ is even.

(b) Let $k(\neq0)$ be the coefficient of $kx^{i}$ in $\rm{dp}(v)$. Therefore $v$ has exactly $k$ neighbor $v_{1}, \ldots , v_{k}$ of degree $i$. Now, $v_{1}, \ldots, v_{k}$ are distinct from $v$, and
\[sc(\rm{dp}(v_{1}))= \ldots =sc(\rm{dp}(v_{k}))=i.\]

(c) Let $\{a_{1}, \ldots , a_{s}\}$ be the set of odd vertices of $G$. $\sum^{n}_{j=1}\deg(a_{j})$ is even; That is $\sum^{s}_{j=1}sc(\rm{dp}(a_{j}))$ is even. Thus $\sum^{s}_{j=1}sec(\rm{dp}(a_{j}))+\sum^{s}_{j=1}soc(\rm{dp}(a_{j}))$ is even. For each $a_{j}$, $1\leq j\leq n$, if $a_{j^{\prime}}$ be an odd neighbor of $a_{j}$, then $a_{j}$ is an odd neighbor of $a_{j^{\prime}}$, as well. Therefore the edge between $a_{j}$ and $a_{j^{\prime}}$, acts two times in calculating $\sum^{s}_{j=1}soc(\rm{dp}(a_{j}))$, once in $soc(\rm{dp}(a_{j}))$ and again in $soc(\rm{dp}(aj^{\prime}))$. Thus $\sum^{s}_{j=1}soc(\rm{dp}(a_{j}))$ is even and thus $\sum^{s}_{j=1}sec(\rm{dp}(a_{j}))$, that is $\sum_{sc(f_{j}) \ is \ odd} sec(f_{j})$, is an even integer.

The argument for second part, is similar, with this difference that we should start with the set of all even vertices, $\{b_{1}, \ldots , b_{t}\}$.
\end{proof}
\begin{example}
The sequence
\[s_{1}=(2x, x^{2}, x, x, x)\]
of polynomials, satisfies (a) and (b), but not (c); The sequence
\[s_{2}=(2x, x^{2}, x^{2}, x, x, x)\]
satisfies (b) and (c), but not (a); Finally the sequence
\[s_{3}=(2x^{2}, x, x, x, x)\]
satisfies (a) and (c), but not (b). Therefore upon Theorem 3.10, the sequences $s_{1}, s_{2}$ and $s_{3}$ is not realizable. Meanwhile it is possible that a non-increasing sequence $q=(f_{1}, f_{2}, \ldots , f_{n})$ of nonzero polynomials with coefficients in nonnegative integers, satisfies (a), (b) and (c), but it do not be realizable yet. Consider for example, the sequence
\[2x^{2}, 2x, 2x, x, x.\]
Note that if the sequence be realizable, then the vertex $v$ which $\rm{dp}(v)=2x$, should be adjacent only with two vertices with degree polynomials $x$ and $x$ (name a and b). But in this case, the degree polynomial of a and b will not be $x$.
\end{example}
Now we study the behavior of degree polynomial under graph operations
\begin{thm}
Let $G$ and $H$ be two simple graphs with disjoint vertex sets, and $u$ be a vertex in $G$. Then
\[\rm{dp}_{G\vee H}(u)=x^{n_{2}}\rm{dp}_{G}(u)+x^{n_{1}}\rm{dp}(H),\]
where $n_{1}$ and $n_{2}$ are the orders of $G$ and $H$, respectively.
\end{thm}
\begin{proof}
If $u$ be an isolated vertex in $G$, then $\rm{dp}_{G}(u)=0$. In this case, since $u$ is adjacent with all vertices of $H$ in $G\vee H$ (by definition of $G\vee H$) and $u$ is not adjacent with any vertex of $G$, if $H$ has $t_{0}$ vertices of degree 0, $t_{1}$ vertices of degree 1, ... , $t_{\Delta}$ vertices of degree $\Delta$ ($\Delta$ is the maximum degree of $H$), then the neighbors of $u$ in $G\vee H$, are restricted to:\\

\ \ \ \ \ \ \ \ \ \ \ \ \ \ \ \ \ \ \ \ \ \ \ \ \ $t_{0}$ vertices\ of\ degree $n_{1}+0,$\\

\ \ \ \ \ \ \ \ \ \ \ \ \ \ \ \ \ \ \ \ \ \ \ \ \ $t_{1}$ vertices\ of\ degree $n_{1}+1,$ 
\vskip -0.45 true cm
\ \ \ \[.\]
\vskip -0.7 true cm
\ \ \ \[.\]
\vskip -0.7 true cm
\ \ \ \[.\] \ \ \ \ \ \ 
\vskip -0.42 true cm
\ \ \ \ \ \ \ \ \ \ \ \ \ \ \ \ \ \ \ \ \ \ \ \ \ $t_{\Delta}$ vertices\ of\ degree $n_{1}+\Delta,$\\

\hskip -0.4 true cm and therefore the degree polynomial of $u$ in $G\vee H$ is:
\[t_{0} x^{n_{1}}+ t_{1} x^{n_{1}+1}+ \ldots+ t_{\Delta} x^{n_{1}+\Delta}= x^{n_{1}} \rm{dp}(H),\]
and therefore the conclusion holds.

Now let $u$ is not isolated vertex in $G$. Suppose that $\rm{dp}_{G}(u)=\sum_{s=1}^{k}c_{i_{s}}x^{i_{s}}$ where $c_{i_{s}}$'s are positive integers and $i_{s}$'s are the distinct degrees of neighbors of $u$ is $G$. It means that the neighbors of $u$ in $G$, are restricted to:
\vskip +0.45 true cm
\ \ \ \ \ \ \ \ \ \ \ \ \ \ \ \ \ \ \ \ \ \ \ \ \ \ \ \ $c_{i_{1}}$ vertices\ of\ degree $i_{1},$ 
\vskip -0.45 true cm
\ \ \ \ \ \ \[.\]
\vskip -0.7 true cm
\ \ \ \ \ \ \[.\]
\vskip -0.7 true cm
\ \ \ \ \ \ \[.\] \ \ \ \ \ \ 
\vskip -0.35 true cm
\ \ \ \ \ \ \ \ \ \ \ \ \ \ \ \ \ \ \ \ \ \ \ \ \ \ \ \ $c_{i_{k}}$ vertices\ of\ degree $i_{k},$\\

\hskip -0.4 true cm Now by definition of $G\vee H$, $u$ will be adjacent in $G\vee H$, with all of the above vertices, and also with any vertex in $H$. Thus the neighbors of $u$ in $G\vee H$ are restricted to:
\vskip +0.45 true cm
\ \ \ \ \ \ \ \ \ \ \ \ \ \ \ \ \ \ \ \ \ \ \ \ \ \ $c_{i_{1}}$ vertices\ of\ degree $n_{2}+i_{1},$
\vskip -0.45 true cm
\ \ \ \[.\]
\vskip -0.7 true cm
\ \ \ \[.\]
\vskip -0.7 true cm
\ \ \ \[.\]
\vskip 0.1 true cm
 \ \ \ \ \ \ \ \ \ \ \ \ \ \ \ \ \ \ \ \ \ \ \ \ \ \ $c_{i_{k}}$ vertices\ of\ degree $n_{2}+i_{k},$\\ 

\ \ \ \ \ \ \ \ \ \ \ \ \ \ \ \ \ \ \ \ \ \ \ \ \ \ $t_{0}$ vertices\ of\ degree $n_{1}+0,$\\

\ \ \ \ \ \ \ \ \ \ \ \ \ \ \ \ \ \ \ \ \ \ \  \ \ \ $t_{1}$ vertices\ of\ degree $n_{1}+1,$  
\vskip -0.45 true cm
\ \ \ \[.\]
\vskip -0.7 true cm
\ \ \ \[.\]
\vskip -0.7 true cm
\ \ \ \[.\] \ \ \ \ \ \ 
\vskip -0.42 true cm
\ \ \ \ \ \ \ \ \ \ \ \ \ \ \ \ \ \ \ \ \ \ \  \ \ \ $t_{\Delta}$ vertices\ of\ degree $n_{1}+\Delta,$\\
\\
where $\rm{dp}(H)=\sum_{i=0}^{\Delta}t_{i}x^{i}$. Therefore
\[\rm{dp}(G\vee H)=c_{i_{1}}x^{n_{2}+i_{1}}+ \ldots+c_{i_{k}} x^{n_{2}+i_{k}}+t_{0}x^{n_{1}}+t_{1}x^{n_{1}+1}+ \ldots+ t_{\Delta}x^{n_{1}+\Delta}=\]
 $x^{n_{2}}\rm{dp}_{G}(u)+x^{n_{1}}\rm{dp}(H)$.
\end{proof}
\begin{rem}
Since for every two simple graphs $G$ and $H$, $G\vee H=H\vee G$, the above theorem, in practice, provides a tool for calculating the degree polynomial of any vertex of $G\vee H$.
\end{rem}
\begin{thm}
If $G$ and $H$ be two simple graphs, and $u$ and $v$ be vertices of $G$ and $H$, respectively, then
\[\rm{dp}_{G\times H}((u,v))= x^{\deg u}\rm{dp}(v)+x^{\deg v} \rm{dp}(u).\]
\end{thm}
\begin{proof}
If $u$ in $G$ and $v$ in $H$ be isolated, then by definition of $G\times H, \ (u,v)$ in $G\times H$ is an isolated vertex and therefore $\rm{dp}((u,v))=0$. On the other hand, $\rm{dp}(u)=0$ and $\rm{dp}(v)=0$. Therefore the conclusion holds.

If $u$ be isolated in $G$ but $v$ dose not be isolated in $H$, supposing $\rm{dp}(v)=\sum_{r_{j}\neq 0}r_{j}x^{j}$ where $r_{j}$'s are positive integers and $j$'s are the disjoint degrees of neighbors of $v$ in $H$, by definition of $G\times H$, since $u$ has not any adjacent in $G$, each neighbor of $(u,v)$ in $G\times H$, is in the form $(u,v^{\prime})$ where $v^{\prime}\sim v$. Meanwhile the degree of such $(u,v^{\prime})$ in $G\times H$, is $\deg u+\deg v^{\prime}$. Since for each $j$, $v$ has $r_{j}$ neighbors of degree $j$, the number of neighbors of $(u,v)$ of degree $\deg u+ j$, will be $r_{j}$. Thus
\[\rm{dp}_{G\times H}((u,v))=\sum_{r_{j}}r_{j}x^{\deg u+ j}= x^{\deg u}\rm{dp}(v).\]
But in this case, $\rm{dp}(u)=0$ and therefore the conclusion holds.

The argument in the case that $v$ is isolated but $u$ is not, is similar to the argument in the recent case.

Now let no one of $u$ and $v$ is not isolated. suppose that $\rm{dp}(u)=\sum_{s=1}^{k}c_{i_{s}}x^{i_{s}}$, and $\rm{dp}(v)=\sum_{t=1}^{k^{\prime}}r_{j_{t}}x^{j_{t}}$, where $c_{i_{s}}$'s and $r_{j_{t}}$'s are positive integers, and $i_{s}$'s and $j_{t}$'s are the disjoint degrees of neighbors of $u$ and $v$, respectively. This means that the neighbors of $u$ in $G$, are restricted to:
\vskip +0.45 true cm
\ \ \ \ \ \ \ \ \ \ \ \ \ \ \ \ \ \ \ \ \ \ \ \ \ \ \ \ $c_{i_{1}}$ vertices\ of\ degree $i_{1},$ 
\vskip -0.45 true cm
\ \ \ \ \ \ \[.\]
\vskip -0.7 true cm
\ \ \ \ \ \ \[.\]
\vskip -0.7 true cm
\ \ \ \ \ \ \[.\] \ \ \ \ \ \ 
\vskip -0.35 true cm
\ \ \ \ \ \ \ \ \ \ \ \ \ \ \ \ \ \ \ \ \ \ \ \ \ \ \ \ $c_{i_{k}}$ vertices\ of\ degree $i_{k},$\\

\hskip -0.4 cm and the neighbors of $v$ in $H$, are restricted to:
\vskip +0.45 true cm
\ \ \ \ \ \ \ \ \ \ \ \ \ \ \ \ \ \ \ \ \ \ \ \ \ \ \ \ $r_{j_{1}}$ vertices\ of\ degree $j_{1},$ 
\vskip -0.45 true cm
\ \ \ \ \ \ \[.\]
\vskip -0.7 true cm
\ \ \ \ \ \ \[.\]
\vskip -0.7 true cm
\ \ \ \ \ \ \[.\] \ \ \ \ \ \ 
\vskip -0.35 true cm
\ \ \ \ \ \ \ \ \ \ \ \ \ \ \ \ \ \ \ \ \ \ \ \ \ \ \ \ $r_{j_{k^{\prime}}}$ vertices\ of\ degree $j_{k^{\prime}}.$\\

\hskip -0.4 cm By definition of $G\times H$, the adjacent vertices of $(u,v)$ in $G\times H$, are of two kinds below:\\
(i) the vertices in the form $(u,b)$ where $b$ is an adjacent of $v$ in $H$,\\
(ii) the vertices in the form $(a,v)$ where $a$ is an adjacent of $u$ in $G$.\\
Since in all vertices of kind (i), $u$ is fixed, such vertices are restricted to:
\vskip +0.45 true cm
\ \ \ \ \ \ \ \ \ \ \ \ \ \ \ \ \ \ \ \ \ \ \ $r_{j_{1}}$ vertices\ of\ degree $\deg u+j_{1},$ 
\vskip -0.45 true cm
\ \ \ \ \ \ \[.\]
\vskip -0.7 true cm
\ \ \ \ \ \ \[.\]
\vskip -0.7 true cm
\ \ \ \ \ \ \[.\] \ \ \ \ \ \ 
\vskip -0.35 true cm
\ \ \ \ \ \ \ \ \ \ \ \ \ \ \ \ \ \ \ \ \ \ \ $r_{j_{k^{\prime}}}$ vertices\ of\ degree $\deg u+j_{k^{\prime}}.$\\

\hskip -0.4 cm Also, since in the vertices of kind (ii), $v$ is fixed, such vertices are restricted to:
\vskip +0.45 true cm
\ \ \ \ \ \ \ \ \ \ \ \ \ \ \ \ \ \ \ \ \ \ \ $c_{i_{1}}$ vertices\ of\ degree $i_{1}+\deg v,$ 
\vskip -0.45 true cm
\ \ \ \ \ \ \[.\]
\vskip -0.7 true cm
\ \ \ \ \ \ \[.\]
\vskip -0.7 true cm
\ \ \ \ \ \ \[.\] \ \ \ \ \ \ 
\vskip -0.35 true cm
\ \ \ \ \ \ \ \ \ \ \ \ \ \ \ \ \ \ \ \ \ \ \ $c_{i_{k}}$ vertices\ of\ degree $i_{k}+\deg v.$\\

\hskip -0.4 cm Not that the degree of each vertex, $(x,y)$, in $G\times H$ is $\deg_{G}x+\deg_{H}y$. Therefore
\[\rm{dp}_{G\times H}((u,v))=(r_{j_{1}}x^{\deg u+j_{1}}+ \ldots+ r_{j_{k^{\prime}}} x^{\deg u+ j_{k^{\prime}}})\]
\[+\ (c_{i_{1}}x^{i_{1}+\deg v}+ \ldots+ c_{i_{k}} x^{i_{k}+\deg v})=\ x^{\deg u}\rm{dp}(v)\]
\hskip -0.05 cm $\ \ \ \ \ \ \ \ \ \ \ \ \ \ \ +\ x^{\deg v} \rm{dp}(u).$
\end{proof}
\begin{thm}
If $G$ and $H$ be two simple graphs, and $u$ and $v$ be vertices of $G$ and $H$, respectively, then
\[\rm{dp}_{G\otimes H}((u,v))= \rm{dp}(u)\otimes\rm{dp}(v)\]
\end{thm}
\begin{proof}
If at least one of $u$ and $v$ be isolated, then by 
definition of $G\otimes H$, $(u,v)$ is an isolated vertex in the graph $G\otimes H$, and therefore $\rm{dp}((u,v))=0$. On the other hand, in this case, at least on of $\rm{dp}(u)$ and $\rm{dp}(v)$ is zero. Therefore by definition of tensor product of polynomials, $\rm{dp}(u)\otimes \rm{dp}(v)$ is zero as well, and the conclusion holds.

Now let no one of $u$ and $v$ be isolated. Suppose that $\rm{dp}(u)=\sum_{i}c_{i}x^{i}$ and $\rm{dp}(v)=\sum_{j}r_{j}x^{j}$, where $c_{i}$'s and $r_{j}$'s are positive integers, and $i$'s and $j$'s are the disjoint degrees of neighbors of $u$ and $v$, respectively. By definition of $G\otimes H$, first, each neighbor of $(u,v)$ in $G\otimes H$, is in the form $(u^{\prime},v^{\prime})$ where $u^{\prime}$ is a neighbor of $u$, and $v^{\prime}$ is a neighbor of $v$, and secondly if $u^{\prime}$ is a neighbor of $u$ of degree $i$ and $v^{\prime}$ is a neighbor of $v$ of degree $j$, then $(u^{\prime},v^{\prime})$ is a neighbor of $(u,v)$ of degree $i.j$. Since for each $i$, $u$ has exactly $c_{i}$ neighbors of degree $i$, and for each $j$, $v$ has exactly $r_{j}$ neighbors of degree $j$, the number $c_{i}r_{j}$ is calculated in the coefficient of $x^{i.j}$ in the degree polynomial of $(u,v)$. This implies that
\[\rm{dp}_{G\otimes H}((u,v))=\rm{dp}(u)\otimes \rm{dp}(v)\]
by definition of $\rm{dp}(u)\otimes \rm{dp}(v)$.
\end{proof}
\begin{thm}
If $G$ and $H$ be two simple graphs and $u$ and $v$ be vertices of $G$ and $H$, respectively, then
\[\rm{dp}_{G[H]}((u,v))=(\rm{dp}(u))^{\curlywedge\times n_{2}}\rm{dp}(H)+ x^{(\deg u)n_{2}}\rm{dp}(v),\]
in which $n_{2}$ is the order of $H$.
\end{thm}
\begin{proof}
If $u$ and $v$, both be isolated, by definition of $G[H]$, $(u,v)$ is isolated in $G[H]$, and therefore $\rm{dp}_{G[H]}((u,v))=0$. But in this case, both $\rm{dp}(u)$ and $\rm{dp}(v)$ are zero polynomials, and therefore the conclusion holds.

If $u$ is isolated in $G$, but $v$ is not isolated in $H$, supposing $\rm{dp}(v)=\sum_{t=1}^{k^{\prime}} r_{j_{t}}x^{j_{t}}$, in which $r_{j_{t}}$'s are positive integers and $j_{t}$'s are the disjoint degrees of neighbors of $v$ in $H$, the neighbors of $v$ in $H$, are restricted to:
\vskip +0.45 true cm
\ \ \ \ \ \ \ \ \ \ \ \ \ \ \ \ \ \ \ \ \ \ \ \ \ \ \ $r_{j_{1}}$ vertices\ of\ degree $j_{1},$ 
\vskip -0.45 true cm
\ \ \ \ \ \ \[.\]
\vskip -0.7 true cm
\ \ \ \ \ \ \[.\]
\vskip -0.7 true cm
\ \ \ \ \ \ \[.\] \ \ \ \ \ \ 
\vskip -0.35 true cm
\ \ \ \ \ \ \ \ \ \ \ \ \ \ \ \ \ \ \ \ \ \ \ \ \ \ $r_{j_{k^{\prime}}}$ vertices\ of\ degree $j_{k^{\prime}}.$\\

\hskip -0.4 cm Since $u$ has no any neighbor in $G$, by definition of $G[H]$, every neighbors of $(u,v)$ is in the form $(u,b)$ with degree $(\deg u)n_{2}+\deg b$, where $b$ is a neighbor of $v$ in $H$. Hence since $u$ is fixed, the neighbors of $(u,v)$ are restricted to:
\vskip +0.45 true cm
\ \ \ \ \ \ \ \ \ \ \ \ \ \ \ \ \ \ \ \ \ $r_{j_{1}}$ vertices\ of\ degree $(\deg u)n_{2}+j_{1},$ 
\vskip -0.45 true cm
\ \ \ \ \ \ \[.\]
\vskip -0.7 true cm
\ \ \ \ \ \ \[.\]
\vskip -0.7 true cm
\ \ \ \ \ \ \[.\] \ \ \ \ \ \ 
\vskip -0.35 true cm
\ \ \ \ \ \ \ \ \ \ \ \ \ \ \ \ \ \ \ \ $r_{j_{k^{\prime}}}$ vertices\ of\ degree $(\deg u)n_{2}+j_{k^{\prime}}.$\\

\hskip -0.4 cm Thus
\[\rm{dp}_{G[H]}((u,v))=r_{j_{1}}x^{(\deg u)n_{2}+j_{1}}+\ldots+r_{j_{k^{\prime}}}x^{(\deg u)n_{2}+j_{k^{\prime}}}=x^{(\deg u)n_{2}}\rm{dp}(v),\]
and since $\rm{dp}(u)=0$, the conclusion in this case holds.
         
If $v$ is isolated in $H$, but $u$ is not isolated in $G$, since $v$ has not any neighbor in $H$, each    neighbor of $(u,v)$ is in the form $(a,b)$ such that $a$ is a neighbor of $u$, and $b$ is a vertex of $H$, and the degree of $(a,b)$ in $G[H]$ is $(\deg a)n_{2}+\deg b$. Suppose that $\rm{dp}(u)=\sum_{s=1}^{k}c_{i_{s}} x^{i_{s}}$ where $c_{i_{s}}$'s are positive integers and $i_{s}$'s are the distinct degrees of neighbors of $u$ in $G$. Supposing $\rm{dp}(H)=\sum_{p=0}^{\Delta} l_{p}x^{p}$ where $l_{p}$'s are the number of vertices of $H$, each one of degree $p$, and $\Delta$ is the maximum degree of $H$, for each $p$, the neighbors of $(u,v)$ whose second components are of degree $p$, are restricted to:
\vskip +0.45 true cm
\ \ \ \ \ \ \ \ \ \ \ \ \ \ \ \ \ \ \ \ \ \ \ \ $l_{p}c_{i_{1}}$ vertices\ of\ degree $i_{1}n_{2}+p,$ 
\vskip -0.45 true cm
\ \ \ \ \ \ \[.\]
\vskip -0.7 true cm
\ \ \ \ \ \ \[.\]
\vskip -0.7 true cm
\ \ \ \ \ \ \[.\] \ \ \ \ \ \ 
\vskip -0.35 true cm
\ \ \ \ \ \ \ \ \ \ \ \ \ \ \ \ \ \ \ \ \ \ \ \ $l_{p}c_{i_{k}}$ vertices\ of\ degree $ i_{k}n_{2}+p.$\\

\hskip -0.4 cm Therefore each $l_{p}c_{i_{s}}$ is calculated in the coefficient of $x^{i_{s}n_{2}+p}$ in $\rm{dp}((u,v))$. Thus
\[\rm{dp}_{G[H]}((u,v))=\sum_{p=0}^{\Delta}\sum_{s=1}^{k} l_{p}c_{i_{s}}x^{i_{s}n_{2}+p}=
\sum_{s=1}^{k}c_{i_{s}}x^{i_{s}n_{2}}\sum_{p=0}^{\Delta}l_{p}x^{p}\]
\[\hskip -2.8 cm =\sum_{s=1}^{k}c_{i_{s}}x^{i_{s}n_{2}}\rm{dp}(H)=
(\rm{dp}(u))^{\curlywedge n_{2}}\rm{dp}(H),\]
and since $\rm{dp}(v)=0$, the conclusion holds.

Now let no one of $u$ and $v$ be isolated. Suppose that $\rm{dp}(u)=\sum_{s=1}^{k}x_{i_{s}}x^{i_{s}}$ and $\rm{dp}(v)=\sum_{t=1}^{k^{\prime}}r_{j_{t}}x^{j_{t}}$ where $c_{i_{s}}$'s and $r_{j_{t}}$'s are positive integers and $i_{s}$'s and $j_{s}$'s are the distinct degree of neighbors of $u$ and $v$, respectively. The adjacent vertices of $(u,v)$ in $G[H]$ are of two kinds below:\\
(i) the vertices in the form $(a,b)$ where $a$ is a neighbor of $u$ in $G$, and $b$ is a vertex of $H$,\\
(ii) the vertices in the form $(u,b)$ where $b$ is a neighbor of $v$ in $H$.\\
Let $\rm{dp}(H)=\sum_{p=0}^{\Delta}l_{p}x^{p}$ where $\Delta$ be the maximum degree of $H$. For each $p$, the neighbors of $(u,v)$ of kind (i) whose second components are of degree $p$, are restricted to:
\vskip +0.45 true cm
\ \ \ \ \ \ \ \ \ \ \ \ \ \ \ \ \ \ \ \ \ \ \ \ $l_{p}c_{i_{1}}$ vertices\ of\ degree $i_{1}n_{2}+p,$ 
\vskip -0.45 true cm
\ \ \ \ \ \ \[.\]
\vskip -0.7 true cm
\ \ \ \ \ \ \[.\]
\vskip -0.7 true cm
\ \ \ \ \ \ \[.\] \ \ \ \ \ \ 
\vskip -0.35 true cm
\ \ \ \ \ \ \ \ \ \ \ \ \ \ \ \ \ \ \ \ \ \ \ \ $l_{p}c_{i_{k}}$ vertices\ of\ degree $ i_{k}n_{2}+p.$\\

\hskip -0.4 cm Therefore each $l_{p}c_{i_{s}}$ is calculated in the coefficient of $x^{i_{s}n_{2}+p}$. On the other hand, since in all neighbors of kind (ii), $u$ is fixed, such vertices are restricted to:
\vskip +0.45 true cm
\ \ \ \ \ \ \ \ \ \ \ \ \ \ \ \ \ \ \ \ \ $r_{j_{1}}$ vertices\ of\ degree $(\deg u)n_{2}+j_{1},$ 
\vskip -0.45 true cm
\ \ \ \ \ \ \[.\]
\vskip -0.7 true cm
\ \ \ \ \ \ \[.\]
\vskip -0.7 true cm
\ \ \ \ \ \ \[.\] \ \ \ \ \ \ 
\vskip -0.35 true cm
\ \ \ \ \ \ \ \ \ \ \ \ \ \ \ \ \ \ \ \ $r_{j_{k^{\prime}}}$ vertices\ of\ degree $(\deg u)n_{2}+j_{k^{\prime}}.$\\

\hskip -0.4 cm Thus 
\[\rm{dp}_{G[H]}((u,v))=\sum_{p=0}^{\Delta}\sum_{s=1}^{k} l_{p}c_{i_{s}} x^{i_{s}n_{2}+p}+\sum_{t=0}^{k^\prime}r_{j_{t}}x^{(\deg u)n_{2}+j_{t}}\]
\[\hskip -3 cm =(\rm{dp}(u))^{\curlywedge n_{2}}\rm{dp}(H)+x^{(\deg u)n_{2}}\rm{dp}(v).\]
\end{proof}
As we saw above, having the degree polynomial sequence of graphs $G$ and $H$ without access to $G$ and $H$, the degree polynomial sequences of $G\vee H,\ G\times H,\ G\otimes H$, and $G[H]$ are calculatable. The following theorem shows that the degree polynomial sequence of the complement of a graph can be calculated having the degree polynomial sequence of that graph without access to the graph itself.
\begin{thm}
Let $G$ be a simple graph and $u$ be a vertex of $G$. Then
\[\rm{dp}_{G^{c}}(u)=(\rm{dp}(G)-\rm{dp}_{G}(u)- x^{\deg_ {u}G})^{\curlywedge (n-1)-}\]
where $n$ is the order of $G$.
\end{thm} 
\begin{proof}
For each $i(\in\mathbb{Z})\geq 0$, the coefficient of $x^{i}$ in $\rm{dp}(G)$, is the total number of the vertices of $G$, each one of degree $i$, and the coefficient of the same $x^{i}$ in $\rm{dp}(u)$, is exactly the number of the vertices of $G$, each one of degree $i$, which are adjacent to $u$ in $G$. Therefore the coefficient of each $x^{i}$ in the polynomial:
\[\rm{dp}(G)-\rm{dp}_{G}(u)\]
is the number of the vertices of degree $i$ (in $G$) which are non-adjacent to $u$. Since $u$ itself is non-adjacent to $u$, the coefficient of $x^{i}$ in the polynomial:
\[\rm{dp}(G)-\rm{dp}_{G}(u)-x^{\deg _{G}u}\]
is exactly the number of the vertices of $G$, other than $u$, which are of degree $i$ and non-adjacent to $u$ (in $G$), and by definition of $G^{c}$, this number is exactly the number of the vertices of $G^{c}$ which are of degree $(n-1)-i$ and adjacent with $u$ (in $G^{c}$). Therefore for each $i$, the coefficient of $x^{(n-1)-i}$ in the recent polynomial, equals exactly the number of the vertices of degree $i$ in $G^{c}$, which are adjacent to $u$ (in $G^{c}$). Thus (based on the meaning of the notation $(\rm{dp}(G)-\rm{dp}_{G}(u)- x^{\deg_{G}u})^{\curlywedge (n-1)-}$) the coefficient of $x^{i}$ in
\[(\rm{dp}(G)-\rm{dp}_{G}(u)- x^{\deg_{G}u})^{\curlywedge (n-1)-}\]
 is exactly the number of the neighbors of $u$ in $G^{c}$, which their degree is $i$.
\end{proof}

\vskip 0.8 true cm

\section{\bf Some open problems}
\vskip 0.4 true cm

Many new questions and open problems can arise from above topics. Some of them are:

(1) How can one characterize all degree polynomial sequences?

(2) Classify degree polynomial sequences of connected graphs and trees.

(3) Characterize all graphs whose degree polynomial sequences formed by polynomials with one term.

(4) Characterize all degree polynomial sequences which realize uniquely. 

\vskip 0.4 true cm




\end{document}